\documentclass[11pt]{amsart}

\usepackage{epigamath}

\usepackage[english]{babel}

\usepackage{amssymb}
\usepackage{amscd}
\usepackage{mathtools}
\usepackage{url}

\usepackage{xcolor}

\numberwithin{equation}{section}

\newtheorem{thm}{Theorem}[section]
\newtheorem{lem}[thm]{Lemma}
\newtheorem{prop}[thm]{Proposition}

\newtheorem{conj}[thm]{Conjecture}

\theoremstyle{definition}
\newtheorem{dfn}[thm]{Definition}
\newtheorem{ntn}[thm]{Notation}

\theoremstyle{remark}
\newtheorem{rmk}[thm]{Remark}
\newtheorem{exa}[thm]{Example}

\newcommand{\Z}{\mathbb{Z}}
\newcommand{\Q}{\mathbb{Q}}

\newcommand{\C}{\mathbb{C}}
\newcommand{\F}{\mathbb{F}}
\newcommand{\A}{\mathbb{A}}
\newcommand{\R}{\mathbb{R}}
\DeclareMathOperator{\Aut}{Aut}
\newcommand{\V}{\mathbb{V}}
\newcommand{\Vla}{\mathbb{V}_{\lambda}}

\DeclareMathOperator{\PGL}{PGL}
\DeclareMathOperator{\SL}{SL}
\DeclareMathOperator{\Sp}{Sp}
\DeclareMathOperator{\Sym}{Sym}
\DeclareMathOperator{\Tr}{Tr}

\newcommand{\Mm}{\mathcal M}
\newcommand{\Aa}{\mathcal A}

\newcommand{\Mmb}{\overline{\mathcal M}}
\newcommand{\Xx}{\mathcal X}
\newcommand{\s}{\mathbb S}

\DeclareMathOperator{\disc}{disc}
\DeclareMathOperator{\cusp}{cusp}
\newcommand{\pd}{\Pi_{\disc}}
\newcommand{\pc}{\Pi_{\cusp}}
\newcommand{\Zhat}{\hat{\Z}}

\newcommand{\Ll}{\mathbb L}

\DeclareMathOperator{\Gal}{Gal}
\DeclareMathOperator{\ch}{ch}
\DeclareMathOperator{\Log}{Log}
\DeclareMathOperator{\Exp}{Exp}
\DeclareMathOperator{\GSp}{GSp}

\newcommand{\h}{\mathrm h}

\newcommand{\supth}[1]{\ensuremath{#1^{\mathrm{th}}}}

\EpigaVolumeYear{7}{2023} \EpigaArticleNr{20} \ReceivedOn{November 15, 2022}
\InFinalFormOn{June 2, 2023}
\AcceptedOn{June 21, 2023}

\title{Cohomology of moduli spaces \\ via a result of Chenevier and Lannes}
\titlemark{Cohomology of moduli spaces}

\author{Jonas Bergstr\"om}
\address{Matematiska Institutionen, Stockholms Universitet, SE-106 91 Stockholm, Sweden}
\email{jonasb@math.su.se}

\author{Carel Faber}
\address{Mathematisch Instituut, Universiteit Utrecht,
P.O. Box 80010, 3508 TA Utrecht, The Netherlands}
\email{C.F.Faber@uu.nl}

\authormark{J.~Bergstr\"om and C.~Faber}

\AbstractInEnglish{We use a classification result of Chenevier and Lannes for algebraic automorphic representations together with a conjectural correspondence with $\ell$-adic absolute Galois representations to determine the Euler characteristics (with values in the Grothendieck group of such representations) of  $\Mmb_{3,n}$ and $\Mm_{3,n}$ for $n \leq 14$ and of local systems $\Vla$ on $\Aa_3$ for $|\lambda| \leq 16$.}



\begin{document}

%%%%%%%%%%%%%%%%%%%%%%%%%%%%%%%
% Title page
%%%%%%%%%%%%%%%%%%%%%%%%%%%%%%%

%\removeabove{}
%\removebetween{}
%\removebelow{}

\maketitle

\begin{prelims}

\DisplayAbstractInEnglish

\bigskip

\DisplayKeyWords

\medskip

\DisplayMSCclass

%\bigskip

%\languagesection{Fran\c{c}ais}

%\bigskip

%\DisplayTitleInFrench

%\medskip

%\DisplayAbstractInFrench

\end{prelims}

%%%%%%%%%%%%%%%%%%%%%
% Table of Contents
%%%%%%%%%%%%%%%%%%%%%

\newpage

\setcounter{tocdepth}{1}

\tableofcontents

%%%%%%%%%%%%%%%%%%%%%
% Content begins here
%%%%%%%%%%%%%%%%%%%%%

\begin{section}{Introduction}
The central role in this paper is played by a classification result of Chenevier and Lannes \cite{ChLa}. For $n\ge1$, they classify the level $1$ algebraic cuspidal automorphic representations of $\PGL_n$ of motivic weight $w\le22$. There are $11$ of them, given in Theorem~\ref{thm-CL}.

As a part of the web of conjectures in the Langlands program, there is the so-called Langlands correspondence, which says, loosely speaking, that there should be a bijection between algebraic automorphic representations and irreducible geometric Galois representations. The assumption that such a bijection exists, formulated explicitly for weight at most $22$ in Conjecture~\ref{conj-gal}, has very concrete geometric consequences, as exemplified by Theorem~\ref{thm-X}. Namely, the cohomology groups of Deligne--Mumford stacks that are smooth and proper over $\Z$ of relative dimension at most $22$ would all consist of Galois representations corresponding to the automorphic representations given by Chenevier and Lannes.

We stress the strong implications this has for several moduli spaces of central interest in algebraic geometry. For example, assuming Conjecture~\ref{conj-gal}, we determine the cohomology groups of the moduli spaces $\Mmb_{3,n}$ of stable $n$-pointed curves of genus~$3$ for all $n\leq14$, as Galois representations endowed with an $\s_n$-action (up to semisimplification).  Similarly, we determine the $\s_n$-equivariant Euler characteristic of $\Mm_{3,n}$ for all $n\leq14$ as an element of the Grothendieck group of Galois representations, and we also determine the Euler characteristics of all standard symplectic local systems of weight at most 16 on the moduli space $\Aa_3$ of principally polarized abelian threefolds.  For this, we use the fact that the integer-valued Euler characteristics are known, and we also use computer counts over small finite fields (with at most 25 elements) of all smooth curves of genus~3 with their zeta functions.  The results are available at \url{https://github.com/jonasbergstroem/M3A3interp}.  This adds to a series of recent results about the cohomology of moduli spaces of pointed curves; see for instance \cite{BFP,CaLa,CLP}.
\end{section}

\subsection*{Acknowledgments}
We thank Ga\"etan Chenevier and the referee for their helpful comments.

%%%%%%%
\begin{section}{A result of Chenevier and Lannes} \label{sec-chenevierlannes}
Let $G$ be a semisimple group scheme over $\Z$. Recall that the
space $\Aa^2(G)$ of square integrable automorphic forms for $G$
may be seen as the space of square integrable functions
on $G(\Q)\backslash G(\A)/G(\Zhat)$
for a natural Radon measure.

Denote by $\Aa_{\disc}(G)$ the discrete part of $\Aa^2(G)$, the closure of the sum of the closed and topologically irreducible $G(\R)$-subrepresentations of $\Aa^2(G)$.  Let $H(G)$ be the Hecke ring of $G$ (\textit{cf.}~\cite[Section~V.2.5]{ChLa}).  One denotes by $\pd(G)$ the set of (isomorphism classes of) closed and topologically irreducible subrepresentations of $\Aa_{\disc}(G)$, for the natural actions of $G(\R)$ and $H(G)$ (\textit{cf.}~\cite[Section~IV.3]{ChLa}).
The elements of $\pd(G)$ are called the \emph{discrete automorphic representations} of~$G$. The set $\pd(G)$ is in general countably infinite.

In addition, one has the subspace $\Aa_{\cusp}(G)\subset \Aa^2(G)$ of cuspidal automorphic forms; it is known to be contained in $\Aa_{\disc}(G)$.  Correspondingly, one has the subset $\pc(G)\subset\pd(G)$ of \emph{cuspidal automorphic representations} of $G$.

Let $\pi=\pi_{\infty}\otimes\pi_f\in\pc(G)$.  Denote by $c_{\infty}(\pi)$ the infinitesimal character of $\pi_{\infty}$ (\textit{cf.}~\cite[Section~VI.3.4]{ChLa}). By the Harish--Chandra isomorphism, it may be seen as a semisimple conjugacy class of the Lie algebra $\hat{\mathfrak g}$ of the Langlands dual group $\hat{G}$ of $G_{\C}$ (\textit{cf.}~\cite[Sections~VI.1.2 and~VI.3.1]{ChLa}).

Now take $G=\PGL_n$. Then $\hat{G}(\C)\cong\SL_n(\C)$.  For $\pi\in\pc(\PGL_n)$, the \emph{weights} of $\pi$ are the eigenvalues of the semisimple conjugacy class $c_{\infty}(\pi)\subset M_n(\C)$. One says that $\pi$ is \emph{algebraic} if the weights of $\pi$ are contained in $\tfrac12\Z$ and if all the differences $w-w'$ between the weights of $\pi$ are integers (\textit{cf.}~\cite[Section~VIII.2.6]{ChLa}).  The \emph{motivic weight} of an algebraic $\pi\in\pc(\PGL_n)$ is by definition twice the largest weight of $\pi$.

The classification result of Chenevier and Lannes \cite[Th\'eor\`eme F]{ChLa} reads as follows.

\begin{thm} \label{thm-CL} Let $n\ge1$, and let $\pi\in\pc(\PGL_n)$ be an algebraic cuspidal automorphic representation of motivic weight at most $22$. Then $\pi$ is one of the following $11$ representations: 
$$1,\hphantom{,} \Delta_{11},\hphantom{,} \Delta_{15},\hphantom{,} \Delta_{17},\hphantom{,} \Delta_{19},\hphantom{,} \Delta_{19,7}, \hphantom{,} \Delta_{21},\hphantom{,} \Delta_{21,5},\hphantom{,} \Delta_{21,9},\hphantom{,} \Delta_{21,13},\hphantom{,} \Sym^2\Delta_{11}.$$
\end{thm}

Here, the weights of $\Delta_w$ are $\pm\tfrac{w}2$, and those of $\Delta_{w,v}$ are $\pm\tfrac{w}2,\pm\tfrac{v}2$. More precisely, for $w\in\{11,15,17,19,21\}$, the representation $\Delta_w$ is the element of $\pc(\PGL_2)$ generated by the $1$-dimensional space $S_{w+1}(\SL_2(\Z))$ of cusp forms of weight $w+1$ for $\SL_2(\Z)$; the representations $\Delta_{19,7}$, $\Delta_{21,5}$, $\Delta_{21,9}$, and $\Delta_{21,13}$ in $\pc(\PGL_4)$ correspond, respectively, to the $1$-dimensional spaces of Siegel cusp forms for $\Sp_4(\Z)$ of types $\Sym^6\det^8$, $\Sym^4\det^{10}$, $\Sym^8\det^8$, and $\Sym^{12}\det^6$; and $\Sym^2\Delta_{11}\in\pc(\PGL_3)$ is the symmetric square of $\Delta_{11}$, with weights $-11,0,11$; \textit{cf.}~\cite[Section~IX.1.3]{ChLa}.

Further strong results along these lines, as well as a simplification of the proof of Theorem~\ref{thm-CL}, have been obtained by Chenevier and Ta\"{\i}bi~\cite{ChTa}. See also the earlier work by Chenevier and Renard \cite{ChRe}.
\end{section}

\begin{section}{Galois representations and cohomology} \label{sec-galois}
Fix a prime $\ell$. By a Galois representation, we will mean an $\ell$-adic representation of $\Gal(\overline{\Q}/\Q)$. A Galois representation will be called geometric if it is unramified for \emph{all} primes $p\neq \ell$ and crystalline at $p=\ell$.  The cyclotomic character will be denoted by $\Ll$. It is of dimension $1$, Hodge--Tate weight $1$, and motivic weight $2$. For an irreducible geometric Galois representation $V$ of dimension $N$, we have that $\wedge^N V=\Ll^i$ for some $i$, and we define the motivic weight $w(V)$ to be $2i/N$. For a semisimple $V$, we let $w(V)$ denote the maximum of the weights of its irreducible pieces.  Let $F_p$ denote a Frobenius element at $p$.

For a Hecke eigenform $f$ in $S_k(\SL_2(\Z))$ (or in $S_{j,k}(\Sp_4(\Z))$), let $a_p(f)$ denote the eigenvalue of the Hecke operator $T_p$. By \cite{Deligne} (respectively, \cite{Laumon} and \cite{Weissauer2}),
there is a $2$-dimensional (respectively, $4$-dimensional) semisimple 
geometric Galois representation $V_f$ such that 
$$\Tr\left(F_p,V_f\right)=a_p(f),$$ and with Hodge--Tate weights $0,w(V)=k-1$ (respectively, $0,k-2,j+k-1,w(V)=j+2k-3$).  It is irreducible, except when $j=0$ and $f$ is a Saito--Kurokawa lift.  Also, $\Sym^2 V_f$, for an eigenform $f$ in $S_k(\SL_2(\Z))$, will be geometric and irreducible; see for instance \cite[Lemma 4.7]{Ramakrishnan}.  Hence, to each of the $11$ automorphic representations given in Theorem~\ref{thm-CL}, there is naturally associated a Galois representation $\phi_i$ for $1 \leq i \leq 11$. Note that all these representations are self-dual in the sense that $\phi_i \cong \phi_i^{\vee} \otimes \det \, \phi_i$.

\begin{ntn}
If $f_1,\ldots,f_N$ is any basis of eigenforms of $S_k(\SL_2(\Z))$ (or of $S_{j,k}(\Sp_4(\Z))$), let $S[k]$ (respectively, $S[j,k]$) denote the direct sum of the Galois representations $V_{f_i}$ for $1 \leq i \leq N$. 

This means in particular that 
\begin{align*}
  &\phi_2=S[12],\, \phi_3=S[16],\, \phi_4=S[18],\,\phi_5=S[20],\,  \phi_6=S[6,8],\,\phi_7=S[22], \\
  &\phi_8=S[4,10],\,\phi_9=S[8,8],\,\phi_{10}=S[12,6],\,\phi_{11}=\Sym^2 S[12],
\end{align*}
which aligns with the notation used in \cite{BFvdG2}
(note that $N=1$ in all these cases).
\end{ntn}

As a part of the web of conjectures in the Langlands program, there is the so-called Langlands correspondence, which says, loosely speaking, that there should be a bijection between algebraic automorphic representations and irreducible geometric Galois representations. In our case the Galois representations corresponding to the automorphic representations have been constructed. Moreover, a semisimple Galois representation is determined by the traces of Frobenius for almost all primes. Hence, what we need is something similar to \cite[Conjecture 3.5]{Taylor}, saying that to an irreducible geometric Galois representation there corresponds an algebraic automorphic representation such that the traces of Frobenius equal the Hecke eigenvalues for all primes. On the basis of such a conjecture, we are in turn led to conjecture the following. Note that this conjecture is in no way original.

\begin{conj} \label{conj-gal} Every semisimple geometric Galois representation of motivic weight at most $22$, and with non-negative Hodge--Tate weights, is a direct sum of representations $\phi_j$ for $1\leq j \leq 11$ tensored with $\Ll^k$ for some $k \geq 0$.
\end{conj}

Let $\Xx$ be a separated Deligne--Mumford stack of finite type over $\Z$ of relative dimension $d$ together with an action of the symmetric group $\s_n$. For a field $k$, let $H^i_c(\Xx \otimes \bar k)$ denote the $\supth{i}$ compactly supported $\ell$-adic \'etale cohomology group. It is an $\ell$-adic representation of $\Gal(\bar k/k) \times \s_n$. We denote the category of such by $\Gal^{\s_n}_k$ and the corresponding Grothendieck group by $K_0(\Gal^{\s_n}_k)$. Denote by $\Gal_k$ and $K_0(\Gal_k)$ the versions without the $\s_n$-action. We identify an irreducible representation of $\s_n$ with a partition $\mu$ of $n$ in the usual way, and we denote the corresponding Schur polynomial by $s_{\mu}$. So we have a decomposition
$$H^i_c(\Xx \otimes \bar k) = \bigoplus_{\mu \vdash n} \left(W_{i,\mu} \boxtimes s_{\mu} \right) \in \Gal^{\s_n}_k$$
for certain $W_{i,\mu}$ in $\Gal_k$; we put $H^i_{c,\mu}(\Xx \otimes \bar k):=W_{i,\mu}$.  Then we define
$$e_{c,\mu}(\Xx \otimes \bar k)= \sum_{i=0}^{2d} (-1)^i \left[H^i_{c,\mu}\left(\Xx \otimes \bar k\right)\right] \in K_0(\Gal_k)$$
and 
$$e_c(\Xx \otimes \bar k)= \sum_{\mu \vdash n} e_{c,\mu}(\Xx \otimes \bar k) \, s_{\mu} \in K_0\left(\Gal^{\s_n}_k\right).$$
Note that for any closed substack $\mathcal Z \subset \mathcal X$, defined over $k$ and invariant under $\s_n$, it holds that
\begin{equation} \label{eq-add}
  e_{c,\mu}\left(\Xx \otimes \bar k\right)=e_{c,\mu}\left(\mathcal Z \otimes \bar k\right)+e_{c,\mu}\left(\left(\Xx \setminus \mathcal Z\right) \otimes \bar k\right).  
\end{equation}
The analogous statements hold for the integer-valued Euler characteristics $E_c$ and $E_{c,\mu}$.

When $\mathcal X$ is smooth, the semisimplification of the Galois representation $H^i_c(\Xx \otimes \overline \Q)$ will be of motivic weight between $0$ and $i$.  If $\Xx$ is also proper, then it will be geometric and pure of weight $i$ (meaning that all of its irreducible pieces will have weight $i$). It follows that $e_{c}(\Xx \otimes \overline \Q)$ as an element of $K_0(\Gal^{\s_n}_{\Q})$ determines $H^i_c(\Xx \otimes \overline \Q)$ as an element of $K_0(\Gal^{\s_n}_{\Q})$ for all $i$. Moreover, by Poincar\'e duality, we then have  that
\begin{equation} \label{eq-poincare}
 \Ll^d\, H^i_c\left(\Xx \otimes \overline \Q\right)^{\vee} \cong H^{2d-i}_c\left(\Xx \otimes \overline \Q\right)  
\end{equation}
for all $0 \leq i \leq d$.

For a finite field $k=\F_q$ with $q=p^r$ and $\ell \neq p$, we define the (geometric) Frobenius $F_q \in \Gal(\bar k/k)$ to be the inverse of $x \mapsto x^q$. We choose a Frobenius element (by abuse of notation, we will use the same letter) $F_{q} \in \Gal(\overline \Q/\Q)$, using an element of the Galois group of the $p$-adic completion of $\Q$ that is mapped to the Frobenius element $F_q \in \Gal(\bar k/k)$.  If $\mathcal X$ is smooth and proper, then for all $q$ such that $p \neq \ell$, these Frobenii satisfy
\begin{equation} \label{eq-frobenii}
\Tr \left(F_q,e_{c,\mu}\left(\Xx \otimes \overline{\Q}\right)\right) = \Tr\left(F_q,e_{c,\mu}\left(\Xx \otimes \bar k\right)\right) \end{equation}
(by \cite[Proposition 3.1]{edixhovenbogaart}) and these traces will be elements of $\Z$. Note that the traces of $F_q$ for (almost) all primes $q$ will determine $e_{c,\mu}(\Xx \otimes \overline{\Q})$ as an element of $K_0(\Gal_{\Q})$, by a Chebotarov density argument.

\begin{dfn} 
Let $\mathbf \Phi_i$ denote the submonoid of $K_0(\Gal_{\Q})$ generated by
$\Ll^{k} \,  \phi_j$ for all $k\geq 0$ and $1\leq j \leq 11$ such that $2k+w(\phi_j)= i$. 
\end{dfn} 

\begin{exa} The following $34$ Galois representations generate $\cup_{i=0}^{22} \mathbf \Phi_{i} \,$:
\begin{align*}
    & 1,\,  \Ll,\, \Ll^2,\, \Ll^3,\, \Ll^4,\, \Ll^5,\, \Ll^6,\, \Ll^7,\, \Ll^8,\, \Ll^9,\, \Ll^{10},\, \Ll^{11},\, S[12],\, \Ll\,\,S[12],\, \Ll^2\,\,S[12],\, \Ll^3\,\,S[12], \\
    &\Ll^4\,S[12],\,  \Ll^5\,S[12],\, S[16],\, \Ll\,S[16],\, \Ll^2\,S[16],\, \Ll^3\,S[16],\, S[18],\, \Ll\,S[18],\,  \Ll^2\,S[18], \\
    &S[20],\, \Ll\,S[20],\, S[22],\, S[6,8],\, \Ll\,S[6,8],\, S[4,10],\, S[8,8],\, S[12,6],\, \Sym^2S[12].
\end{align*}\end{exa}

\begin{thm} \label{thm-X}
Assume that Conjecture~\ref{conj-gal} holds. Let $\mathcal X$ be a smooth and proper Deligne--Mumford stack over $\Z$ of relative dimension $d$ at most~$22$ with a given action of the symmetric group $\mathbb S_n$. Then the semisimplification of $H_{c,\mu}^i( \mathcal X \otimes \overline{\Q})$ is an element of\, $\mathbf \Phi_i$ for any $0 \leq i \leq d$ and partition $\mu$ of $n$.
\end{thm}

\begin{proof}
  For every $i \geq 0$, the cohomology group $H_{c,\mu}^i( \mathcal X \otimes \overline{\Q})$  with $\mathcal X$ a smooth and proper Deligne--Mumford stack over $\Z$ is geometric and pure of weight $i$; see \cite[Section~1, p.~79]{Taylor}  and also \cite{edixhovenbogaart} for the generalization to Deligne--Mumford stacks.  The result for $0 \leq i \leq d$ now follows from Conjecture~\ref{conj-gal}.
\end{proof}

\begin{rmk} By \eqref{eq-poincare} it follows that the $\Ll^k \phi_j$  that occur in $H_{c,\mu}^i( \mathcal X \otimes \overline{\Q})$ for $i>d$ are also in $\mathbf \Phi_i$ with the additional constraint 
$k+w(\phi_j) \leq d$. 
\end{rmk}

\begin{dfn}
Let $\mathbf \Psi_i$ denote the submonoid of $K_0(\Gal_{\Q})$ generated by $\Ll^{k} \,  \phi_j$ for all $k\geq 0$ and $1\leq j \leq 11$ such that $k+w(\phi_j) \leq i$. 
\end{dfn}
\end{section}

\begin{rmk} Fix any $i \geq 0$, and assume that $H_{c}^{2i}(\mathcal X  \otimes \overline{\Q})$, as an element of $K_0(\Gal_{\Q})$, is a direct sum of representations $\Ll^i$. By \cite[Theorem 2.1]{edixhovenbogaart} it then follows that $H_{c}^{2i}(\mathcal X  \otimes \overline{\Q})$ is semisimple. 
\end{rmk}

\begin{section}{Moduli spaces}
\begin{subsection}{Moduli spaces of curves} \label{sec-Mgn} 
For any $g,n$, with $n \geq 3$ if $g=0$ and $n \geq 1$ if $g=1$, let $\Mm_{g,n}$ and $\Mmb_{g,n}$ denote the moduli spaces of smooth, respectively stable, $n$-pointed curves of genus $g$. These are smooth Deligne--Mumford stacks over $\Z$ of dimension $3g-3+n$, and $\Mmb_{g,n}$ is also proper over $\Z$.  They come with a natural action of the symmetric group $\mathbb S_n$, permuting the $n$ marked points.

The complement of $\Mm_{g,n}$ in $\Mmb_{g,n}$, denoted by $\partial \Mm_{g,n}$, can be written as a disjoint union of locally closed subsets
\begin{equation} \label{eq-stratM}
\Mm_\Gamma:=\left( \prod_{v \in \Gamma} \Mm_{g(v),n(v)} \right)\,/{\Aut \, \Gamma},
\end{equation}
one for each non-trivial stable graph $\Gamma$ of genus $g$ and with $n$ markings (see, \textit{e.g.}, \cite[Proposition~XII.10.11]{geometryofcurves}).

Note that 
\begin{equation} \label{eq-dimstrat}
  \dim \Mm_\Gamma=\sum_{v \in \Gamma} \left(3g(v)-3+n(v)\right)<3g-3+n.  
\end{equation}

We will analyze which elements of $K_0(\Gal_{\Q})$ can appear in $e_{c,\mu}(\Mm_{g,n}\otimes \overline{\Q})$ using a formula by Getzler and Kapranov, namely \cite[Theorem 8.13]{GK}. We will follow the notation of \cite{GK}, but with a slight generalization. Put $G=\Gal_k$ for some field $k$.  We will consider a stable $\mathbb S$-module $\mathcal V=\{\mathcal V((g,n)): g,n \geq 0\}$, see \cite[Section 2.1]{GK}, but where the graded pieces $\mathcal V((g,n))_i$ will be finite-dimensional $\ell$-adic $G \times \mathbb S_n$-representations (instead of just $\mathbb S_n$-representations). These pieces will then be direct sums of $G \times \mathbb S_n$-representations of the form $W \boxtimes V$. 
Let $[W]$ denote the element of the Grothendieck group $K_0(G)$
of finite-dimensional $G$-representations
corresponding to a $G$-representation $W$.
For an element $W'=\sum_j a_j [W_j] \in K_0(G)$, with $a_j \in \Z$ and $G$-representations $W_j$ with characters $\chi_{W_j}$, define the character $\chi_{W'}=\sum_j a_j \chi_{W_j}$. Note that this gives a bijection between elements of $K_0(G)$ and their characters. 

Let $\Lambda$ be the ring of symmetric power series; see \cite[Section 7.1]{GK}. We define the character $\ch_n(W \boxtimes V)$ to be $\chi_W \cdot \ch_n(V)$, compare with~\cite[Section 7]{GK}, which is an element of $\Lambda$ with coefficients in $K_0(G)$. With this definition, $\mathbb{C}\h(\mathcal V)$, see \cite[Section 8.2]{GK}, will be an element of $\Lambda((h))$ with coefficients in $K_0(G)$. We then define $\mathbb{M} \mathcal V$ as in \cite[Section 2.17]{GK}.

 Let $p_i$ denote the $\supth{i}$ power sum in $\Lambda$. For a $G$-representation $W$, we then put $p_n \circ (\chi_W \cdot p_k)= \chi_{p_n \circ W} \cdot p_{kn}$, where $\chi_{p_n \circ W}(g)=\chi_{W}(g^n)$. Theorem 8.13 of \cite{GK} then follows, with the same proof, and so
 \begin{equation} \label{eq:getzlerkapranov}
 \mathbb{C}\h(\mathbb{M} \mathcal V)=\Log(\Exp(\Delta) \, \Exp(\mathbb{C}\h(\mathcal V))).
 \end{equation}

Putting $\mathcal V((g,n))_i=H^i_c(\Mm_{g,n} \otimes \overline{\Q})$, which are finite-dimensional $\ell$-adic $G \times \mathbb S_n$-representations, we get a stable $\mathbb S$-module $\mathcal V$. For a partition $\mu \vdash n$, we let $s_{\mu}$ denote the corresponding Schur polynomial.  We now have, see \cite[Section 6.2]{GK},
 \[
\mathbb{C}\h(\mathcal V)=\sum_{2(g-1)+n>0} h^{g-1} \,\sum_{\mu \vdash n} e_{c,\mu}\left(\Mm_{g,n}\otimes \overline{\Q}\right) \,  s_{\mu}
\]
and
\[
\mathbb{C}\h(\mathbb{M} \mathcal V)=\sum_{2(g-1)+n>0} h^{g-1} \,\sum_{\mu \vdash n} e_{c,\mu}\left(\Mmb_{g,n}\otimes \overline{\Q}\right) \,  s_{\mu}.
\]

\begin{prop} \label{prop-Mgn} Assume that Conjecture~\ref{conj-gal} holds.
For any $g, n$ such that $3g-3+n \leq 22$ and partition $\mu \vdash n$,  $e_{c,\mu}(\Mm_{g,n}\otimes \overline{\Q})$ is in $\mathbf \Psi_{3g-3+n}$. 
\end{prop}

\begin{proof} We use induction on the dimension. The statement clearly holds for the zero-dimensional base case $\Mmb_{0,3}=\Mm_{0,3}$. Take any $g,n$ such that $d:=3g-3+n \leq 22$, and assume that $e_{c,\mu}(\Mm_{\tilde g,\tilde n}\otimes \overline{\Q}) \in \mathbf \Psi_{3\tilde g-3+\tilde n}$ for any $\tilde g, \tilde n$ such that $3 \tilde g-3+\tilde n < d$. Using the Getzler--Kapranov formula, we find that $e_{c,\mu}(\partial \Mm_{g,n}\otimes \overline{\Q})$  is in $\mathbf \Psi_{d-1}$, noting that no non-trivial tensor product of representations $\phi_{i}$ with $2 \leq i \leq 11$ can appear, for dimension reasons. 

Applying Theorem~\ref{thm-X} to $\Mmb_{g,n}$, we see that if $\Ll^k \phi_j$ appears in $H^i_c(\Mmb_{g,n} \otimes \overline{\Q})$ for $i \leq d$ with $w=w(\phi_j)$, then its Poincar\'e dual \eqref{eq-poincare} will have weight $2d-i=w+2(d-w-k)$ and be of the form $\Ll^{d-w-k} \phi_j$. Hence $(d-w-k)+w=d-k \leq d$, and so $e_{c,\mu}(\Mmb_{g,n} \otimes \overline{\Q})$ is in $\mathbf \Psi_{d}$. We conclude by applying \eqref{eq-add}.
\end{proof}
\end{subsection}

\begin{subsection}{Moduli spaces of abelian varieties}\label{sec-Ag}
Let $\Aa_g$ denote the moduli space of principally polarized abelian varieties of dimension $g$.  Let $\pi\colon \Xi_g \to \Aa_g$ denote the universal abelian variety, and define the standard local system by $\V:=R^1\pi_* \Q_{\ell}$.  For a highest weight $\lambda=(\lambda_1, \lambda_2,\ldots,\lambda_g)$ with $\lambda_i \geq \lambda_{i+1}$ for $1 \leq i \leq g-1$ and $\lambda_g\geq 0$, we get an irreducible representation of the symplectic group $\GSp_{2g}$ and a corresponding local system $\Vla$, in such a way that $\V$ corresponds to the dual of the standard representation. Put $|\lambda|:=\lambda_1+\lambda_2+\ldots+\lambda_g$.

\begin{thm}[\protect{Faltings--Chai, \cite[Theorem~VI.5.5]{FaCh}, \textit{cf.}~\cite[Theorem~17]{GetTRR}}] \label{thm-FaCh}
The Hodge--Tate weights of $H^i_c(\Aa_g \otimes \overline \Q,\Vla)$ belong to the set
\begin{equation} \label{eq-weights}
W_{\lambda}:=\left\{\sum_{s \in S} s: S \subset {\{ \lambda_g+1,\lambda_{g-1}+2,\ldots,\lambda_1+g  \} } \right \}, 
\end{equation}
consisting of at most $2^g$ distinct elements.  
\end{thm}

\begin{rmk} \label{rmk-wtAg} 
Also note  that the cohomology groups $H^i_c(\Aa_g \otimes \overline \Q,\Vla)$ will have motivic weight between $0$ and $|\lambda|+i$; see \cite[Corollaire 3.3.4]{DeligneWeilII}.
\end{rmk}

\begin{dfn}
Let $\mathbf \Psi_{\lambda}$ denote the submonoid of $K_0(\Gal_{\Q})$ generated by $\Ll^{k} \,  \phi_j$ for all $k\geq 0$ and $1\leq j \leq 11$ such that 
$2k+w(\phi_j)\leq g(g+1)+|\lambda|$ and such that
$k+w_j$ is in $W_{\lambda}$, for any Hodge--Tate weight $w_j$ of $\phi_j$.
\end{dfn}

\begin{exa} The following Galois representations generate $\Psi_{(10,0,0)}\,$:
$$1,\,\Ll,\,\Ll^2,\,\Ll^3,\,\Ll^{2} \, S[12],\,\Ll^3 \, S[12],\,S[16],\, \Ll\, S[16] .$$
The following Galois representations generate $\Psi_{(10,4,2)}\,$:
$$1,\,\Ll^3,\,\Ll^6,\,\Ll^9,\,\Ll^{13}\,,S[20],\,\Ll^3\, S[20],\,S[6,8],\,\Ll^3\,S[6,8].$$
\end{exa}
\end{subsection}

\begin{subsection}{The Torelli map} \label{sec-torelli}
The moduli spaces we are considering are related via the Torelli map $t_g\colon \Mm_{g}\to \Aa_g$ for $g \geq 2$ and $t_1\colon \Mm_{1,1} \to \Aa_1$. This morphism between stacks has degree $1$ if $g \leq 2$ (it is an isomorphism in genus~$1$) and degree $2$ if $g \geq 3$, ramified along the hyperelliptic locus $\mathcal{H}_g \subset \Mm_g$. If $|\lambda|$ is even, then
\begin{equation} \label{eq-equal}
e_{c}\left( t_g\left(\Mm_{g} \otimes \overline \Q\right),\Vla\right)= e_{c}\left( \Mm_{g}\otimes \overline \Q,t_g^*\Vla\right) \in  K_0\left(\Gal_k\right), 
\end{equation} 
where these Euler characteristics (and $E_c$) are defined as in Section~\ref{sec-galois}.  If $|\lambda|$ is odd, then
$$e_{c}\left( t_g\left(\Mm_{g}\otimes \overline \Q\right),\Vla\right)=0$$ 
due to the presence of the automorphism $-1$ on all Jacobians. But for $g \geq 3$, there is no reason for $e_{c}( \Mm_{g}\otimes \overline \Q,t_g^*\Vla)$ to be zero. The first non-zero example is
$$e_{c}\left(\Mm_{3}\otimes \overline \Q,t_3^*\V_{2,1,0}\right)=\Ll^6-\Ll^2;$$
see \cite{JBquart}.

\begin{rmk} In fact, $e_c(\Mm_{g}\otimes \overline \Q,t_g^*\Vla)$ will contain Galois representations that are not present in 
any $e_{c}(\Aa_{h}\otimes \overline \Q,{\mathbb{V}_{\lambda'}})$
with $h=g$ and $|\lambda'|\leq|\lambda|$, or $h<g$.
These could be called \emph{Teichm\"uller motives}.
Conjecturally, the first two such instances occur
for $g=3$ and $\lambda=(11,3,3)$ or $(7,7,3)$ 
and correspond to two of the seven  automorphic representations of weight~$23$ in \cite[Theorem 3]{ChTa}, as described in \cite{CFvdG3}.
\end{rmk} 

\begin{lem} \label{lem-transl}
For each partition $\mu$ of $n$, there are elements $a_{\mu,\lambda} \in \Z[\Ll]$ such that
$$e_{c,\mu}\left( \Mm_{g,n}\right)=\sum_{|\lambda| \leq n} a_{\mu,\lambda} \, 
e_c\left( \Mm_{g},t_g^*\Vla\right).$$
Conversely, for each $\lambda$ with $|\lambda|=n$, there are elements $b_{\lambda,\mu} \in \Z[\Ll]$ such that 
$$e_c\left( \Mm_{g},t_g^*\Vla\right)=\sum_{|\mu| \leq n} b_{\lambda,\mu} \, e_{c,\mu}\left( \Mm_{g,n}\right).$$
\end{lem}

\begin{proof}[Proof sketch.] This is well known.
The Leray spectral sequence associated to the forgetful morphism $\pi\colon \Mm_{g,n} \to \Mm_{g}$ gives a relation between the cohomology of $\Mm_{g,n}$ and the cohomology of the higher direct images $R^i \pi_* \Q_{\ell}$. The latter sheaves can be expressed in terms of local systems $t_g^*\Vla$ for $\lambda$ with $|\lambda| \leq n$. Taking Euler characteristics, this connection becomes a representation-theoretic question of relating $\s_n$-representations to $\GSp_{2g}$-representations. This can be formulated in symmetric polynomials as finding an expression of the (usual) Schur polynomial $s_{\mu}$ in terms of symplectic Schur polynomials $s_{<\lambda>}$ (see the definition in \cite[Appendix A.45]{FuHa}, which should be homogenized using an extra variable $q$). This can always be done, and there are elements $a'_{\mu,\lambda} \in \Z[q]$ such that
$$s_{\mu}=\sum_{|\lambda| \leq n} a'_{\mu,\lambda} \, s_{<\lambda>},$$
and replacing $q$ with $\Ll$ in $a'_{\mu,\lambda}$ gives $a_{\mu,\lambda}$. 
There are also $b'_{\lambda,\mu} \in \Z[q]$ such that
$$s_{<\lambda>}=\sum_{|\mu| \leq n} b'_{\lambda,\mu} \,s_{\mu} ,$$
and replacing $q$ with $\Ll$ in $b'_{\lambda,\mu}$, we get $b_{\lambda,\mu}$. 
\end{proof}
\end{subsection}

\begin{subsection}{The cases $\boldsymbol{\Aa_1}$ and $\boldsymbol{\Aa_2}$}
  From \cite{Deligne} we know that for any even $a>0$, $e_c(\Aa_1 \otimes \overline \Q ,\V_{a})=-1-S[a+2],$ and this formula also works for $a=0$ if we put $S[2]:=-\Ll-1$. If $a \leq 21$, then $\dim S_{a+2}(\SL_2(\Z))=0$, except when $a+2$ equals $12$, $16$, $18$, $20$, or $22$ (in which cases the dimension equals $1$). We conclude the following.
  
\begin{lem} \label{lem-A1phi} We have that $e_c(\Aa_1 \otimes \overline \Q ,\V_{a})$ is in $\mathbf \Psi_{(a)}$ for $a+1 \leq 22$.
\end{lem}

By \cite[Theorem 2.1]{Petersen}, see also \cite{FvdG}, we have a similar formula for $e_c(\Aa_2,\V_{a,b})$ with $S[a-b,b+3]$ as ``main term''.

\begin{lem} \label{lem-A2phi} We have that $e_c(\Aa_2 \otimes \overline \Q ,\V_{a,b})$ is in $\mathbf \Psi_{(a,b)}$ for $a+b +3\leq 22$. 
\end{lem}

\begin{proof} The only $(a,b)$ with $a+b+3 \leq 22$ such that $\dim S_{a-b,b+3}(\Sp_4(\Z)) >0$ are $(11,5)$, $(11,7)$, $(13,5)$, $(15,3)$, which correspond to $\phi_6$, $\phi_8$, $\phi_9$, $\phi_{10}$, respectively, together with $(7,7)$ and  $(9,9)$. The last two spaces are spanned by an eigenform which is a Saito--Kurokawa lift. The corresponding Galois representations are reducible. More precisely, they are $S[0,10]=S[18]+\Ll^9+\Ll^8$ and $S[0,12]=S[22]+\Ll^{11}+\Ll^{10}$, respectively; see for instance \cite{Weissauer2}. The result now follows from \cite[Theorem 2.1]{Petersen}. 
\end{proof}
\end{subsection}

\begin{subsection}{The boundary of $\boldsymbol{\Mm_{3,n}}$} \label{sec-boundary}
Computing $e_{c,\mu}(\Mm_{0,n} \otimes \overline \Q )$ for any $n$ and partition $\mu$ is not difficult; for details see for instance \cite{Getzler}.

Using Lemma~\ref{lem-transl} and the formula for $e_c(\Aa_1 \otimes \overline \Q ,\V_{a})$ described in Section~\ref{sec-Ag}, one can compute the Euler characteristic $e_{c,\mu}(\Mm_{1,n} \otimes \overline \Q )$ for any $n$ and partition $\mu$; see \cite{G-res}.

There is a stratification
\begin{equation} 
\Aa_2 = t_2(\Mm_2)  \; \sqcup \; \Aa_1^{\times 2}/\mathbb S_2.   
\end{equation}
Again, using the formula for $e_c(\Aa_1 \otimes \overline \Q ,\V_{a})$, we can compute $e_c((\Aa_1^{\times 2}/\mathbb S_2) \otimes \overline \Q ,\V_{a,b})$ for any local system $\V_{a,b}$. For details on this computation, see for instance \cite{Petersen_d}. Now, using the formula for $e_c(\Aa_2 \otimes \overline \Q ,\V_{a,b})$ mentioned in Section~\ref{sec-Ag} and Lemma~\ref{lem-transl}, we can also compute $e_{c,\mu}(\Mm_{2,n} \otimes \overline \Q )$ for any $n$ and partition $\mu$.

\begin{rmk} For any $n \leq 9$ and partition $\mu$, $e_{c,\mu}(\Mm_{2,n} \otimes \overline \Q )$ is a polynomial in $\Ll$. Adding the contribution from the boundary, we see that the same statement holds for $\Mmb_{2,n}$.
(Note that $S[12]$ in $H^{11}_c(\Mm_{1,11})$ comes with the alternating representation.) But, one finds $\Ll \, S[12]$ both in $e_{c,[1^{10}]}(\Mm_{2,10} \otimes \overline \Q )$ and in $e_{c,[1^{10}]}(\Mmb_{2,10} \otimes \overline \Q )$.
\end{rmk}

Together, this gives us all the pieces to compute $e_{c,\mu}(\partial \Mm_{3,n} \otimes \overline \Q )$ for any $n$ and partition $\mu$ using the Getzler--Kapranov formula described in Section~\ref{sec-Mgn}.
\end{subsection}

\begin{subsection}{The case $\boldsymbol{\Aa_3}$}
In \cite[Conjecture 7.1]{BFvdG2} there is a conjectural formula, corroborated by extensive computations, for $e_c(\Aa_3\otimes \overline \Q,\V_{a,b,c})$ with ``main term'' $S[a-b,b-c,c+4]$, described below. For any eigenform $f \in S_{j,k,l}(\Sp_6(\Z))$, we conjecture in \cite{BFvdG2} that there is a corresponding $8$-dimensional Galois representation $V_f$ similarly to the cases of $\SL_2(\Z)$ and $\Sp_4(\Z)$. If $f_1,\ldots,f_N$ is any basis of eigenforms of $S_{j,k,l}(\Sp_6(\Z))$, then we let $S[j,k,l]$ denote the direct sum of the Galois representations $V_{f_i}$ for $1 \leq i \leq N$. The only $(a,b,c)$ with $a+b+c+6 \leq 22$ such that $\dim S_{a-b,b-c,c+4}(\Sp_6(\Z)) \neq 0$ is $(8,4,4)$; see \cite{Taibi}. This is conjectured to be a lift with Galois representation
\begin{equation} \label{conj844}
S[4,0,8]=\left(S[12]+\Ll^6+\Ll^5\right)\, S[12].
\end{equation}

These conjectures together tell us that $e_c(\Aa_3\otimes \overline \Q,\V_{a,b,c})$ is in $\mathbf \Psi_{(a,b,c)}$ for $a+b+c+6\leq 22$. We will now arrive at this conclusion assuming  Conjecture~\ref{conj-gal} instead. 

\begin{prop} \label{prop-A3phi} Assume that Conjecture~\ref{conj-gal} holds. Then $e_c(\Aa_3 \otimes \overline \Q ,\V_{a,b,c})$ is in $\mathbf \Psi_{(a,b,c)}$ for $a+b+c+6\leq 22$. 
\end{prop}
\begin{proof} 
There is a stratification
\begin{equation} \label{eq-stratA}
\Aa_3 = t_3(\Mm_3)  \; \sqcup \; t_2(\Mm_2) \times \Aa_1 \;\sqcup \; \Aa_1^{\times 3}/\mathbb S_3.   
\end{equation}
Assume that $a+b+c+6\leq 22$. It follows from Equation~\eqref{eq-equal}, Proposition~\ref{prop-Mgn}, and Lemma~\ref{lem-transl} that $e_c(t_3(\Mm_3) \otimes \overline \Q ,\V_{a,b,c})$ is in $\mathbf \Psi_{a+b+c+6}$. Using the same results, we find that $e_c((t_2(\Mm_2) \times \Aa_1) \otimes \overline \Q ,\V_{a,b,c})$ and $e_c((\Aa_1^{\times 3}/\mathbb S_3) \otimes\overline \Q ,\V_{a,b,c})$ are in $\mathbf \Psi_{a+b+c+6}$. The result now follows from the stratification~\eqref{eq-stratA} and Theorem~\ref{thm-FaCh} together with Remark~\ref{rmk-wtAg}.\end{proof}
\end{subsection}
\end{section}

\begin{section}{Finding linear relations} \label{sec-relations}
For any partition $\mu$ of $n \leq 16$, we know (assuming Conjecture~\ref{conj-gal}) by Theorem~\ref{thm-X} and Poincar\'e duality that there are integers $c^{\mu}_{k,j}$ such that
\begin{equation} \label{eq-munkn}
e_{c,\mu}\left(\Mmb_{3,n} \otimes \overline \Q \right)=\sum_{i=0}^{5+n} \, \sum_{2k+w(\phi_j)=i} c^{\mu}_{k,j} \left(\Ll^k+\Ll^{6+n-k-w(\phi_j)}\right) \, \phi_j +\sum_{2k+w(\phi_j)=6+n} c^{\mu}_{k,j} \Ll^k \, \phi_j.   
\end{equation}

Similarly, for every $\lambda$ such that $|\lambda| \leq 16$, there are integers $d^{\lambda}_{k,j}$ such that
\begin{equation}\label{eq-aunkn}
e_{c}\left(\Aa_{3} \otimes \overline \Q, \Vla \right)=\sum_{k,j} d^{\lambda}_{k,j} \Ll^k \phi_j, 
\end{equation}
where the sum is over $k,j$ such that $\Ll^k \phi_j \in \mathbf \Psi_{\lambda}$.
Taking dimensions we have 
\begin{equation*}
E_{c,\mu}\left(\Mmb_{3,n} \otimes \overline \Q \right)=\sum_{i=0}^{5+n} \, \sum_{2k+w(\phi_j)=i} 2 \, c^{\mu}_{k,j} \, \dim_{\Q_{\ell}}\phi_j +\sum_{2k+w(\phi_j)=6+n} c^{\mu}_{k,j} \dim_{\Q_{\ell}} \phi_j
\in \Z    
\end{equation*}
and 
$$E_{c}\left(\Aa_{3} \otimes \overline \Q, \Vla \right)=\sum_{k,j} d^{\lambda}_{k,j} 
\dim_{\Q_{\ell}}\phi_j\in \Z.$$

\begin{subsection}{Integer-valued Euler characteristic} \label{sec-inteul} 
  The integer-valued Euler characteristics $E_c(\Mm_{3}\otimes \overline \Q,\Vla)$ and $E_c(\Aa_{3}\otimes \overline \Q,\Vla)$ can be computed for any $\lambda$ using an algorithm found in \cite{BvdG}; see Tables 3, 4 and 6 in \textit{op.~cit.}~for some values. Using an analogue of Lemma~\ref{lem-transl}, we can then compute $E_{c,\mu}(\Mm_{3,n} \otimes \overline \Q)$ for any~$n$. By the result in Section~\ref{sec-boundary}, we can, for any $n$ and partition $\mu$, compute $e_{c,\mu}(\partial \Mm_{3,n} \otimes \overline \Q )$ and $E_{c,\mu}(\partial \Mm_{3,n} \otimes \overline \Q )$. (Note that by the work of Gorsky, $E_{c,\mu}(\Mm_{g,n} \otimes \overline \Q)$ can in fact be computed for any $g,n$, and $\mu$; see \cite{Gorsky}.)

 With this we can compute $E_{c,\mu}(\Mmb_{3,n} \otimes \overline \Q)$ (respectively, $E_c(\Aa_{3} \otimes \overline \Q,\Vla)$), which gives a linear relation between the $c^{\mu}_{k,j}$ (respectively, the $d^{\lambda}_{k,j}$) for all $n \leq 16$ (respectively, $|\lambda| \leq 16$).
\end{subsection}

\begin{subsection}{Traces of Frobenius} \label{sec-traces}
Using the computer we have found the necessary information about curves of genus at most $3$ over finite fields $\F_q$ with $q \leq 25$, together with their zeta functions, in order to apply the Lefschetz trace formula to compute $\Tr(F_q,e_{c,\mu}(\Mmb_{3,n}\otimes \overline{\F}_q))$ and $\Tr(F_q,e_c(\Aa_3\otimes \overline{\F}_q,\Vla))$; see \cite[Section 8]{BFvdG2} for more details.

It follows from \eqref{eq-frobenii} and \eqref{eq-munkn} that 
\begin{equation*} \label{eq-trM}
  \Tr\left(F_q,e_{c,\mu}\left(\Mmb_{3,n}\otimes \overline{\F}_q\right)\right) =\sum_{i=0}^{5+n} \, \sum_{2k+w(\phi_j)=i} c^{\mu}_{k,j} \left(q^k+q^{6+n-k-w(\phi_j)}\right) \, \Tr\left(F_q,\phi_j\right)+\sum_{2k+w(\phi_j)=6+n} c^{\mu}_{k,j} q^k \, \Tr\left(F_q,\phi_j\right)
\end{equation*}
for every partition $\mu$ of $n\leq 16$. Also, 
\begin{equation*}\label{eq-trA}
  \Tr\left(F_q,e_c\left(\Aa_3\otimes \overline{\F}_q,\Vla\right)\right)=\sum_{k,j} d^{\lambda}_{k,j} q^k \Tr\left(F_q,\phi_j\right)
\end{equation*}
follows, for every $\lambda$ such that $|\lambda| \leq 16$, from \eqref{eq-aunkn} and \cite[Corollary 3.2]{BvdG} together with the stratification \eqref{eq-stratA}.  The contribution from the boundary can be computed by applying the trace of Frobenius to $e_{c,\mu}(\partial \Mm_{3,n} \otimes \overline \Q )$; see Section~\ref{sec-boundary}.  The equality of Frobenii for $\partial \Mm_{3,n}$ follows from $\partial \Mm_{3,n}=\Mmb_{3,n} \setminus \Mm_{3,n}$ together with \eqref{eq-frobenii} for $\Mmb_{3,n}$ and \cite[Corollary 3.2]{BvdG} together with Lemma~\ref{lem-transl} for $\Mm_{3,n}$.  It is well known how to compute $\Tr(F_q,\phi_j)$ for $q \leq 25$ and $1 \leq j \leq 11$. This gives an additional $14$ linear relations between the $c^{\mu}_{k,j}$ (respectively, the $d^{\lambda}_{k,j} $) for all $n \leq 16$ (respectively, $|\lambda| \leq 16$).
 \end{subsection}

\end{section}

%%%%%%%
\begin{section}{Euler characteristics of local systems on \texorpdfstring{$\boldsymbol{\Aa_3}$}{A\textunderscore 3}}
\begin{exa} 
The following Galois representations generate $\Psi_{(14,2,0)}$:
$$1,\, \Ll,\, \Ll^4,\, \Ll^5,\, S[18],\, \Ll \, S[18],\, \Ll^4\, S[18],\, \Ll^5\, S[18],\, S[22],\, \Ll \, S[22],\, S[12,6],\, \Ll \, S[12,6].$$
Assuming Conjecture~\ref{conj-gal} we have by Proposition~\ref{prop-A3phi} that 
\[
e_{c}\left(\Aa_{3} \otimes \overline \Q, \Vla \right)=
\alpha_1 \, 1+
\alpha_2 \, \Ll+
\alpha_3 \, \Ll^4+
\alpha_4 \, \Ll^5+ 
\alpha_5 \, S[18] +
\ldots + 
\alpha_{12} \, \Ll \, S[12,6]
\]
for some integers $\alpha_1,\ldots,\alpha_{12}$. 
Let us write the first three of the linear relations we found as in Section~\ref{sec-relations}: 
\begin{align*} 
4=E_c\left(\Aa_3,\V_{14,2,0}\right)=&\alpha_1+\alpha_{2}+\alpha_{3}+\alpha_{4}+2\alpha_{5}+\cdots+4\alpha_{11}+4\alpha_{12}, 
\\
-270=\Tr\left(F_2,e_c\left(\Aa_3,\V_{14,2,0}\right)\right)=&\alpha_1+2 \alpha_2+2^4\alpha_3+2^5\alpha_4-528\alpha_5+\cdots
 -240\alpha_{11}+2\cdot(-240) \alpha_{12},\\
67800= \Tr\left(F_3,e_c\left(\Aa_3,\V_{14,2,0}\right)\right)=&\alpha_1+3 \alpha_2+3^4\alpha_3+3^5\alpha_4-4284\alpha_5+\cdots
 + 68040\alpha_{11}+ 3 \cdot 68040\alpha_{12}.
\end{align*}
If we continue and take the first $12$ linear equations and solve, we find that
$$e_c\left(\Aa_3,\V_{14,2,0}\right)=\Ll-\Ll^5+S[12,6].$$
\end{exa}

\begin{thm} \label{thm-A3eul} 
Assume that Conjecture~\ref{conj-gal} and equation~\eqref{conj844} hold. Then \cite[Conjecture 7.1]{BFvdG2} is true for all $\lambda$ such that $|\lambda| \leq 16$.
\end{thm}

\begin{proof}
For all $\lambda$ such that $|\lambda| \leq 16$, $\mathbf \Psi_{\lambda}$ has rank $r_{\lambda}$ at most $12$. In Section~\ref{sec-relations} we found $15$ linear relations. The $r_{\lambda}$ first relations, coming from the integer-valued Euler characteristic together with the traces of Frobenius (ordered by the size of $q$), for the $r_{\lambda}$ unknowns $d^{\lambda}_{k,j}$ in \eqref{eq-aunkn} turn out to be linearly independent. These can therefore be used to determine $e_{c}(\Aa_{3} \otimes \overline \Q, \Vla )$. The results agree with the ones given in \cite[Section~10.2.2, pp.~120--121]{BFvdG2}.
  By \cite{Taibi} it follows that $s_{n(\lambda)}=0$ for all $|\lambda| \leq 16$ except when $\lambda=(8,4,4)$, for which $s_{n(\lambda)}=1$. Applying equation~\eqref{conj844} in the case $\lambda=(8,4,4)$ gives the result.
\end{proof}

\begin{rmk} \label{rem6final} The Euler characteristics $e_{c}(\Aa_{3} \otimes \overline \Q, \Vla )$ are determined unconditionally for all $|\lambda| \leq 6$; see \cite[Theorem 8.1]{BFvdG2}.  
\end{rmk}

\end{section}

%%%%%%%
\begin{section}{The cohomology of \texorpdfstring{$\boldsymbol{\Mmb_{3,n}}$}{Mbar\textunderscore{(3,n)}} for \texorpdfstring{$\boldsymbol{n \leq 14}$}{n at most 14}} 
Compared to the situation for local systems on $\Aa_3$, we have less information about the Hodge--Tate weights of the cohomology groups of $\Mmb_{3,n}$.

\begin{thm} \label{thm-rational} Let $\mathcal X$ be a smooth and proper Deligne--Mumford stack over\, $\Q_{\ell}$ of dimension $d$ with unirational coarse moduli space. For any $m \neq 0,2d$, the $\ell$-adic $\Gal(\overline{\Q}_{\ell}/\Q_{\ell})$-representation 
$H_c^m(\mathcal X \otimes \overline{\Q}_{\ell})$ is crystalline and cannot contain the Hodge--Tate weight $m$.
\end{thm}

\begin{proof} We follow the notation of \cite{Bogaart} and let $D_{dR}$ denote the de Rham Fontaine-functor. By \cite{Faltings,Tsuji}, generalized in \cite[Corollary 8.12]{Bogaart} to Deligne--Mumford stacks, we have an isomorphism of filtered vector spaces 
$$D_{dR}\left(H^m\left(\mathcal X \otimes \overline{\Q}_{\ell}\right)\right) \cong H^m_{dR}\left(\mathcal X/\overline{\Q}_{\ell}\right).$$ This identifies the part of $H^m(\mathcal X \otimes \overline{\Q}_{\ell})$ of Hodge--Tate weight $m$ with $H^0(\mathcal X \otimes \overline{\Q}_{\ell},\Omega^m)$.  Since $H^0(\mathcal X \otimes \overline{\Q}_{\ell},\Omega^m)$ is a birational invariant, and since it vanishes for any projective space, the result follows.
\end{proof} 

\begin{dfn} Let $\mathbf \Phi'_{i}$ denote the submonoid of $K_0(\Gal_{\Q})$ generated by $\Ll^{k} \,  \phi_j$ for all $k>0$ and $1\leq j \leq 11$ such that $2k+w(\phi_j)=i$. Put $ \mathbf \Phi'_{\leq i}:=\mathbf \Phi_{0}\cup_{l=1}^{i} \mathbf \Phi'_{l}$.
\end{dfn}

\begin{thm} \label{thm-M3eul} 
Assume that Conjecture~\ref{conj-gal} is true. For all partitions $\mu$ of $n \leq 14$, $e_{c,\mu}(\Mm_{3,n} \otimes \overline \Q )$ and $e_{c,\mu}(\overline{\Mm}_{3,n} \otimes \overline \Q )$ can be determined.
\end{thm}
\begin{proof} 
It is easy to see that the coarse moduli space of $\Mmb_{3,n} \otimes 
\overline{\Q}$ is rational for $n  =14$; hence it is unirational for $n < 14$.
Using Theorem~\ref{thm-rational} and Conjecture~\ref{conj-gal}, we conclude that $H_c^m(\Mmb_{3,n}\otimes \overline{\Q})$ is in $ \mathbf \Phi'_{m}$ for any $0 <m \leq 6+n$.

Since $\Mmb_{g,n}$ is irreducible, we know that
$$H^{0}_c\left(\Mmb_{g,n} \otimes \overline{\Q}\right)$$
is $1$-dimensional, and it will furthermore have a trivial action by $\Gal(\overline{\Q}/\Q)$. In \cite[Theorem 2.2]{arbarellocornalba} there is a formula for the dimension (using duality) of 
$$H^{2}_c\left(\Mmb_{g,n} \otimes \overline{\Q}\right),$$
and it is shown that this cohomology group is tautological, so it will  consist of a direct sum of Galois representations $\Ll$. 

For any $n \leq 13$, $\mathbf \Phi'_{\leq 6+n}$ has rank $r_n \leq 17$.  Fix any partition $\mu$ of $n$. Since we know the coefficients of $1$ and $\Ll$, we have $r_n-2$ unknowns $c^{\mu}_{k,j}$. In Section~\ref{sec-relations} we found $15$ linear relations for the $r_n$ unknowns $c^{\mu}_{k,j}$, which turn out to be linearly independent. These therefore determine $e_{c,\mu}(\Mmb_{3,n} \otimes \overline \Q )$. Since we can compute $e_{c,\mu}(\partial \Mm_{3,n} \otimes \overline \Q )$ for any $n$, see Section~\ref{sec-boundary}, we can also determine $e_{c,\mu}({\Mm}_{3,n} \otimes \overline \Q )$ for $n \leq 13$.

Since $\mathbf \Phi'_{\leq 20}$ has rank $18$, we are missing one relation for $n=14$. But by Theorem~\ref{thm-A3eul} we know $e_c(\Aa_3,\V_{\lambda})$ for any $\lambda$ such that $|\lambda|=14$. Using the stratification \eqref{eq-stratA} together with the information about genus $1$ and $2$, we can compute $e_{c}(t_3(\Mm_{3}) \otimes \overline \Q,\Vla)$ for any $\lambda$ such that $|\lambda|=14$. Using Lemma~\ref{lem-transl} and the above, we can compute $e_{c}(\Mm_{3} \otimes \overline \Q,t_3^*\Vla)$ for any $\lambda$ such that $|\lambda| \leq 14$. Using Lemma~\ref{lem-transl} again, we can compute $e_{c,\mu}(\Mm_{3,n} \otimes \overline \Q )$ for all $n \leq 14$ and partitions $\mu$. By Section~\ref{sec-boundary} we can also compute $e_{c,\mu}(\partial \Mm_{3,n} \otimes \overline \Q )$ for all $n \leq 14$ and partitions $\mu$.
\end{proof}

\begin{rmk} The Euler characteristics $e_{c,\mu}(\Mm_{3,n} \otimes \overline \Q )$ and $e_{c,\mu}(\overline{\Mm}_{3,n} \otimes \overline \Q )$ are determined unconditionally for $n \leq 7$;  see \cite{JBquart} together with the results of Section~\ref{sec-boundary}.  
In \cite{CaLa} it is proven that $e_{c,\mu}(\overline{\Mm}_{3,n} \otimes \overline \Q )$ is a polynomial in $\Ll$ for $n$ up to 8.  So the result for $e_{c,\mu}(\Mm_{3,n} \otimes \overline \Q )$ and $e_{c,\mu}(\overline{\Mm}_{3,n} \otimes \overline \Q )$ determined above also holds unconditionally for $n=8$.
In turn, this means that
the Euler characteristics $e_{c}(\Aa_{3} \otimes \overline \Q, \Vla )$ are determined unconditionally 
for all $\lambda$ with $|\lambda| = 8$ (\textit{cf.}~Remark~\ref{rem6final}).
\end{rmk}

\begin{rmk} For $8 \leq n \leq 14$, the Euler characteristic $e_{c}(\Mm_{3,n} \otimes \overline \Q )$ is \emph{not} a polynomial in $\Ll$. The same holds for 
$e_{c}(\overline{\Mm}_{3,n} \otimes \overline \Q )$ for $10 \leq n \leq 14$.  Note that this result is unconditional, as can be seen using the interpolation in the proof of Theorem~\ref{thm-M3eul}.
\end{rmk}

\begin{exa} We give two examples assuming Conjecture~\ref{conj-gal}:
\begin{align*}
e_{c,[2^4,1^6]}\left(\Mm_{3,14} \otimes \overline \Q\right) = & S[6,8]+(\Ll^4+6\Ll^3-16\Ll^2-6\Ll+24) S[12]-3\Ll^8+ \\ &+12\Ll^7-60\Ll^6+308\Ll^5 -119\Ll^4-646\Ll^3+192\Ll ^2+281\Ll-55
\end{align*}
and 
\[
e_{c,[2^4,1^6]}\left(\Mmb_{3,14} \otimes \overline \Q\right)=
(-8\Ll^6-31\Ll^5-31\Ll^4-8\Ll^3)S[12].
\]
\end{exa}

\end{section}

%%%%%%%%%%%%%%%%%%%%%
% References
%%%%%%%%%%%%%%%%%%%%%


\begin{thebibliography}{CFvdG20+++}

\bibitem[AC98]{arbarellocornalba}
E.~Arbarello and M.~Cornalba, 
{\em Calculating cohomology groups of moduli spaces of curves via algebraic geometry},
Publ.\ Math.\ Inst.\ Hautes \'Etudes Sci.\ \textbf{88} (1998), 97--127.

\bibitem[ACG11]{geometryofcurves}
E.~Arbarello,~M. Cornalba, and P.~Griffiths, 
{\em Geometry of algebraic curves. Volume II},
Grundlehren math.\ Wiss., vol~268, Springer, Heidelberg, 2011.

\bibitem[Ber08]{JBquart} 
J.~Bergstr\"om
{\em Cohomology of moduli spaces of curves of genus three via point counts}, 
J.~reine angew.\ Math.\ \textbf{622} (2008), 155--187.

\bibitem[BFvdG14]{BFvdG2}
J.~Bergstr\"om, C.~Faber, and G.~van der Geer, 
{\em Siegel modular forms of degree three and the cohomology of local systems}, 
Selecta Math.\ (N.S.) \textbf{20} (2014), no.~1, 83--124.

\bibitem[BFP22]{BFP}
J.~Bergstr\"om, C.~Faber, and S.~Payne, 
{\em Polynomial point counts and odd cohomology vanishing on moduli spaces of stable curves}, preprint \arXiv{2206.07759} (2022).

\bibitem[BvdG08]{BvdG}
J.~Bergstr\"om and G.~van der Geer, 
{\em The Euler characteristic of local systems on the moduli of curves and abelian varieties of genus three}, 
J.~Topol.\ \textbf{1} (2008), no.~3, 651--662.

\bibitem[vdB08]{Bogaart}
T.~van den Bogaart, 
{\em Links between cohomology and arithmetic}, 
Ph.~D.~thesis, Leiden University, 2008, Available from \url{https://math.leidenuniv.nl/scripties/vandenBogaart.pdf}. 

\bibitem[vdBE05]{edixhovenbogaart}
T.~van den Bogaart and B.~Edixhoven,  
{\em Algebraic stacks whose number of points over finite fields is a polynomial}, 
in {\em Number fields and function fields---two parallel worlds}, pp.~39--49,
Progr.\ Math., vol.~239, Birkh\"auser Boston, Boston, MA, 2005. 

\bibitem[CL22]{CaLa}
S.~Canning and H.~Larson, 
{\em On the Chow and cohomology rings of moduli spaces of stable curves},
preprint \arXiv{2208.02357} (2022). 

\bibitem[CLP23]{CLP}
S.~Canning, H.~Larson, and S.~Payne, 
{\em The eleventh cohomology group of $\Mmb_{g,n}$}, 
Forum Math.\ Sigma \textbf{11} (2023), Paper No.~e62.
 
\bibitem[CL19]{ChLa}
G.~Chenevier and J.~Lannes, 
{\em Automorphic forms and even unimodular lattices. Kneser neighbors of Niemeier lattices} (translated from the French by R. Ern\'e), 
Ergeb.\ Math.\ Grenzgeb.~(3), vol.~69, 
Springer, Cham, 2019.

\bibitem[CR15]{ChRe}
G.~Chenevier and D.~Renard, 
{\em Level one algebraic cusp forms of classical groups of small rank}, 
Mem.\ Amer.\ Math.\ Soc.\ \textbf{237} (2015), no.~1121. 

\bibitem[CT20]{ChTa}
G.~Chenevier and O.~Ta\"{\i}bi,  
{\em Discrete series multiplicities for classical groups over $\Z$ 
and level 1 algebraic cusp forms}, 
Publ.\ Math.\ Inst.\ Hautes \'Etudes Sci.\ \textbf{131} (2020), 261--323. 

\bibitem[CFvdG20]{CFvdG3}
F.~Cl\'ery, C.~Faber, and G.~van der Geer, 
{\em Concomitants of ternary quartics and vector-valued Siegel and
Teichm\"uller modular forms of genus three}, 
Selecta Math.\ (N.S.) \textbf{26} (2020), no.~4, Paper No.~55.

\bibitem[Del71]{Deligne}
P.~Deligne,  
{\em Formes modulaires et repr\'esentations $\ell$-adiques},  
in {\em S\'eminaire Bourbaki. Vol. 1968/69: Expos\'es 347--363}, Exp.\ No.\ 355,
pp.~139--172, Lecture Notes in Math., vol.~175, Springer-Verlag, Berlin, 1971.

\bibitem[Del80]{DeligneWeilII}
\bysame,  
{\em La conjecture de Weil. II}, 
Inst.\ Hautes \'Etudes Sci.\ Publ.\ Math.\ \textbf{52} (1980), 137--252.

\bibitem[FvdG04]{FvdG} 
C.~Faber and G.~van der Geer,
{\em Sur la cohomologie des syst\`{e}mes locaux sur les espaces des modules
  des courbes de genre $2$ et des surfaces ab\'{e}liennes, I, II},
C.~R.\ Math.\ Acad.\ Sci.\ Paris \textbf{338} (2004), 381--384, 467--470.

\bibitem[Fal89]{Faltings}
G.~Faltings, 
{\em Crystalline cohomology and $p$-adic Galois-representations},  
in: {\em Algebraic analysis, geometry, and number theory} (Baltimore, MD, 1988) , pp.~25--80, Johns Hopkins Univ.\ Press, Baltimore, MD, 1989. 

\bibitem[FC90]{FaCh}
G.~Faltings and C.-L. Chai,  
{\em Degeneration of abelian varieties} (with an appendix by D.~Mumford),  
Ergeb.\ Math.\ Grenzgeb.~(3), vol.~22, 
Springer-Verlag, Berlin, 1990. 

\bibitem[FH91]{FuHa}
W.~Fulton and J.~Harris,  
{\em Representation theory. A first course},   
Grad.\ Texts in Math.\ \textbf{129}, 
Springer-Verlag, New York, 1991.

\bibitem[Get95]{Getzler} 
E.~Getzler, 
{\em Operads and moduli spaces of genus 0 Riemann surfaces},  
in: {\em The moduli space of curves} (Texel Island, 1994), pp.~199--230,
Progr.\ Math., vol.~129, Birkh\"auser Boston, Boston, MA, 1995. 


\bibitem[Get98]{GetTRR}
\bysame, 
{\em Topological recursion relations in genus 2} 
in: {\em Integrable systems and algebraic geometry} (Kobe/Kyoto, 1997), pp.~73--106, World Sci.\ Publ., River Edge, NJ, 1998.

\bibitem[Get99]{G-res}
\bysame, 
{\em Resolving mixed {H}odge modules on configuration spaces}, 
Duke Math.~J.\ \textbf{96} (1999), no.~1, 175--203.

\bibitem[GK98]{GK}
E.~Getzler and M.\,M.~Kapranov, 
{\em Modular operads}, 
Compos.\ Math.\ \textbf{110} (1998), no.~1, 65--126. 

\bibitem[Gor14]{Gorsky}
E.~Gorsky, 
{\em The equivariant Euler characteristic of moduli spaces of curves},
Adv.\ Math.\ \textbf{250} (2014), 588--595. 

\bibitem[Lau05]{Laumon}
G.~Laumon, 
{\em Fonctions z\^etas des vari\'et\'es de Siegel de dimension trois. 
Formes automorphes. II. Le cas du groupe GSp(4)}, 
Ast\'erisque \textbf{302} (2005), 1--66. 

\bibitem[Pet13]{Petersen_d}
D.~Petersen, 
{\em Cohomology of local systems on loci of $d$-elliptic abelian surfaces}, 
Michigan Math.~J.\ \textbf{62} (2013), no.~4, 705--720. 

\bibitem[Pet15]{Petersen}
\bysame, 
{\em Cohomology of local systems on the moduli of principally polarized
abelian surfaces}, 
Pacific J.~Math.\ \textbf{275} (2015), no.~1, 39--61. 

\bibitem[Ram09]{Ramakrishnan}
D.~Ramakrishnan,
{\em Remarks on the symmetric powers of cusp forms on GL(2)},  
in: {\em Automorphic forms and L-functions I. Global aspects}, pp.~237--256, 
Contemp.\ Math.\ \textbf{488}, Israel Math.\ Conf.\ Proc., 
Amer.\ Math.\ Soc., Providence, RI, 2009. 

\bibitem[Ta\"i17]{Taibi}
O.~Ta\"ibi, 
{\em Dimensions of spaces of level one automorphic forms for split classical groups using the trace formula},   
Ann.\ Sci.\ \'Ec.\ Norm.\ Sup\'er.\ (4) \textbf{50} (2017), no.~2, 269--344. 
  
\bibitem[Tay04]{Taylor}
R.~Taylor, 
{\em Galois representations}, 
Ann.\ Fac.\ Sci.\ Toulouse Math.\ (6) \textbf{13} (2004), no.~1, 73--119. 

\bibitem[Tsu02]{Tsuji}
T.~Tsuji,
{\em Semi-stable conjecture of Fontaine-Jannsen: a survey.
Cohomologies $p$-adiques et applications arithm\'etiques, II}, 
Ast\'erisque \textbf{279} (2002), 323--370. 

\bibitem[Wei05]{Weissauer2}
R.~Weissauer,  
{\em Four dimensional Galois representations.
  Formes automorphes. II. Le cas du groupe GSp(4)},
Ast\'erisque \textbf{302} (2005), 67--150. 
\end{thebibliography}
\end{document}